\documentclass[11pt,a4paper]{amsart}
\usepackage[utf8]{inputenc}
\usepackage[english]{babel}
\usepackage[T1]{fontenc}
\usepackage{amsmath}
\usepackage{amsthm}
\usepackage{amsfonts}
\usepackage{amssymb}
\usepackage{lmodern}
\usepackage{mathrsfs}
\usepackage{physics}
\usepackage[colorlinks=true,citecolor=blue,linkcolor=black]{hyperref}
\usepackage{graphicx}
\usepackage[left=2.5cm,right=2.5cm,top=2.5cm,bottom=2.5cm]{geometry}
\usepackage{enumerate}
\usepackage{tikz-cd}
\usepackage{xcolor}
\usepackage{bm}
\usetikzlibrary{backgrounds}
\usetikzlibrary{calc}
\usetikzlibrary{hobby}
\usetikzlibrary{decorations.markings}
\usetikzlibrary{arrows.meta}
\usetikzlibrary{patterns}
\usepackage{caption}
\usepackage{subcaption}
\usepackage{array}
\usepackage{float}
\usepackage{afterpage}

\numberwithin{equation}{section}

\DeclareMathOperator{\Ram}{Ram}
\DeclareMathOperator{\Irr}{Irr}
\DeclareMathOperator{\Slope}{Slope}

\DeclareMathOperator{\Id}{Id}

\DeclareMathOperator{\Aut}{Aut}

\DeclareMathOperator{\GL}{GL}

\newcommand{\IC}{\mathbb C}

\newcommand{\cir}[1]{\langle #1 \rangle}

\begin{document}

\renewcommand{\proofname}{Proof}
\renewcommand{\Re}{\operatorname{Re}}
\renewcommand{\Im}{\operatorname{Im}}
\renewcommand{\labelitemi}{$\bullet$}

\newtheorem{theorem}{Theorem}[section]
\newtheorem{proposition}[theorem]{Proposition}
\newtheorem{lemma}[theorem]{Lemma}
\newtheorem{corollary}[theorem]{Corollary}
\newtheorem{conjecture}[theorem]{Conjecture}

\theoremstyle{definition}
\newtheorem{definition}[theorem]{Definition}
\newtheorem{example}{Example}[section]
\theoremstyle{remark}
\newtheorem{remark}{Remark}[section]
\newtheorem*{claim}{Claim}


\title{Simplification of exponential factors of irregular connections on $\mathbb P^1$}

\author{Jean Douçot}
\address{"Simion Stoilow" Institute of mathematics of the Romanian Academy, Calea Griviței 21,
010702-Bucharest, Sector 1, Romania}
\email{jeandoucot@gmail.com}

\maketitle

\begin{abstract}
We give an explicit algorithm to reduce the ramification order of any exponential factor of an irregular connection on $\mathbb P^1$, using the same types of basic operations as in the Katz-Deligne-Arinkin algorithm for rigid irregular connections. The exponential factor reached when the algorithm terminates is, up  to admissible deformations, the unique factor with minimal ramification order in the orbit of the initial factor under successive applications of basic operations. Furthermore, we show that for every even integer $n\geq 0$, there is up to admissible deformations a finite number of non-simplifiable exponential factors at infinity such that the corresponding elementary wild character variety has complex dimension $n$, which conjecturally implies that there is a finite number of isomorphism classes of elementary wild character varieties in any dimension. These results can be viewed as saying that the set of all possible level data of exponential factors has the structure of a disjoint union of an infinite number of infinite rooted trees, each tree being associated to a given dimension $n$ and with a finite number of trees for each $n$. In particular, in dimension 2 there is a unique tree, corresponding to the Painlevé I moduli space.
\end{abstract}

\section{Introduction}

\subsection{General context}

In this article, we continue our recent investigations motivated by the question of the classification of wild character varieties in genus zero. 

In previous works \cite{doucot2021diagrams, boalch2022twisted, doucot2024basic, doucot2024moduli}, we defined several combinatorial invariants attached to an arbitrary irregular connection  $(E,\nabla)$ on $\mathbb P^1$. These invariants only depend on the connection through its \emph{formal data} $(\bm\Theta, \bm{\mathcal C})$, which consist of a collection $\bm{\Theta}$ of \emph {exponential factors}, or \emph{Stokes circles} $I_i=\cir{q_i}$, encoding the exponential terms $e^{q_i}$ appearing in the horizontal sections of the connection near its singularities, each Stokes circle coming with an integer multiplicity $n_i\geq 1$, and the datum $\bm{\mathcal C}$ of a conjugacy class $\mathcal C_i\subset \GL_{n_i}(\mathbb C) $ for each Stokes circle $I_i$, its \emph{formal monodromy}.  

The invariants are: a \emph{fission forest} \cite{boalch2022twisted}, characterizing these formal data up to \emph{admissible deformations} \cite{boalch2014geometry, doucot2022local, doucot2023topology}, a \emph{diagram} \cite{doucot2021diagrams}, generalizing several works \cite{okamoto1992painleve,crawley2001geometry,boalch2008irregular, boalch2012simply,hiroe2014moduli,hiroe2017linear,boalch2016global, boalch2020diagrams} establishing links between moduli spaces of connections and  quiver varieties, and a \emph{short fission tree} \cite{doucot2024basic}, which is a refined invariant containing strictly more information than the diagram. \\

On the other hand, there exist various types of natural operations on (irreducible) irregular connections on $\mathbb P^1$. We will consider here the following types of \textit{basic operations}:
\begin{itemize}
\item The Fourier transform;
\item Twists, consisting of tensoring $(E,\nabla)$ by a given rank one connection on $\mathbb P^1$;
\item Möbius transformations.
\end{itemize}

The Fourier transform acts on connections in a complicated way, typically changing their rank, number of singularities, and pole orders. Twists and Möbius transformations act in a simpler way, with twists shifting the exponential factors, and Möbius transformation changing the positions of the singularities, but in combination with the Fourier transform, successive applications of basic operations can modify the formal data of connections in an increasingly intricate manner.

Our invariants behave in various ways with respect to the different basic operations. The diagram and the short fission tree are preserved by the Fourier transform, while the fission forest is preserved by twists and Möbius transformations; on the other hand the diagram is in general not preserved by twists and Möbius transformations, while the fission forest is not preserved by the Fourier transform. One would thus like to understand better how far these invariants can bring us towards classifying genus zero wild character varieties. For instance, using basic operations to reduce the rank of a connection as much as possible, could we hope to define a `minimal' short fission tree that would characterize isomorphism classes of genus zero wild character varieties?\\

A reason for considering this class of operations on connections is that they allow to completely classify (irreducible) irregular connections on $\mathbb P^1$ in the \emph{rigid} case, that is when the wild character variety is just a point: if $(E,\nabla)$ is an irreducible rigid irregular connection on $\mathbb P^1$, it can always be brought to the trivial rank 1 connection by successive application of basic operations. This was first proved (in the $\ell$-adic setting) in the case of regular singularities (corresponding to usual local systems on the Betti side) by Katz \cite{katz1996rigid}, who gave an explicit simplification algorithm, which at each step strictly reduces the rank of the connection, involving the operation of \emph{middle convolution}, which can be obtained as a composition of basic operations. The algorithm was later extended \cite{deligne2006letter,arinkin2010rigid} to the case of irregular singularities, in the complex setting. This naturally raises the following question, which is another motivation for this work: what can we say in the non-rigid case? As mentioned by Arinkin \cite[§5.2]{arinkin2010rigid}, we can run the algorithm in the non-rigid case as well: the process then stops at some rank $>2$ which cannot be reduced further. Can we describe these stopping points? It is also not clear a priori whether the stopping points really correspond to the minimal rank that can be reached from the initial connection by successive applications of basic operations: indeed there might be other ways do ``go down'' than the exact steps prescribed by the algorithm, and when the algorithm terminates it is a priori not excluded that we have just reached a ``local minimum'' for the rank, not a global one. 
 \\

\subsection{Elementary wild character varieties}
These questions seem quite difficult to answer in the general case, for arbitrary connections on $\mathbb P^1$, due to the fact that the Fourier transform may at the same time reduce the contribution to the rank of one Stokes circle, while increasing the contribution of another. In this work, as a first step, we consider the case where there is only one Stokes circle. We show that, in this case, we can set up things so that there is essentially a unique way to ``go down'', and we will never be stuck in a local minimum.

The contribution to the rank of a connection of a Stokes circle $I$ with multiplicity $n\geq 1$ is given by $nr$, where $r=\Ram(I)$ is the \emph{ramification order} of $I$, so for a connection with a single Stokes circle reducing the rank amounts to reducing the ramification order. 

In the \emph{untwisted} case $r=1$, connections on $\mathbb P^1$ with just one singularity and a single untwisted Stokes circle $I$ with multiplicity 1 do not give rise to very interesting moduli spaces: the wild character variety $\mathcal M_B(I,\mathcal C)$ is nonempty only for the trivial conjugacy class $\mathcal C=\{1\}\subset \GL_1(\mathbb C)=\mathbb C^*$ and is then just a point. On the other hand, if $I$ is \textit{twisted}, that is $r>1$, $\mathcal M_B(\Theta,\mathcal C)$ can be nonempty for a unique conjugacy class $\mathcal C=\{1\}$ or $\mathcal C=\{-1\}$ in $\GL_1(\mathbb C)$, and we obtain a nontrivial wild character variety $\mathcal M_B(I):=\mathcal M_B(I,\mathcal C)$.

\begin{definition}
An \emph{elementary wild character variety} is a genus zero wild character variety of the form $\mathcal M_B(I)$, where $I$ is a twisted Stokes circle.
\end{definition}

Up to symplectic algebraic isomorphism, an elementary wild character variety $\mathcal M_B(I)$ doesn't depend on the position of the singularity nor on the exact values of the coefficients of the Stokes circle $I$: it only depends on its equivalence class of \emph{admissible deformations} \cite{boalch2014geometry,boalch2022twisted}, which is characterized by a \emph{level datum} $\mathbf L=L(I)\subset \mathbb Q_{>0}$, so we will allow ourselves to write $\mathcal M_B(\mathbf L)$. The expected dimension of $\mathcal M_B(\mathbf L)$ is an even integer $B_{\mathbf L}$, and in this case the diagram associated to the wild character variety by the construction of \cite{doucot2021diagrams} consists of a single vertex to which are attached $\frac{B_{\mathbf L}}{2}$ loops.

In turn, the main question that we are considering in this article is the classification of elementary wild character varieties.\\

\subsection{Main results} Let $\mathcal S$ denote the set of all possible Stokes circles, lying above all points of $\mathbb P^1$. The crucial fact that we are relying on is that any basic  operation $A$ induces a well-defined self-bijection of $\mathcal S$, that we will also be denoting by $A$. Therefore, it makes sense to study orbits of a single Stokes circle under successive applications of basic operations.  

Let us denote by $\mathcal O$ the group whose elements are compositions of a finite number of basic operations. The group $\mathcal O$ acts on $\mathcal S$, and for $I\in \mathcal S$, we denote by $\mathcal O\cdot I$ its orbit under $\mathcal O$. 

Let us also denote by $\mathbb L$ the set of all possible level data, and for $\mathbf L\in \mathbb L$, let $\mathcal O\cdot \mathbf L \subset \mathbb Q_{>0}$ be the set of level data of Stokes circles that can be obtained by successive application of basic operations from a Stokes circle with level datum $\mathbf L$, that is

\[
\mathcal O\cdot \mathbf L:=\bigcup_{\substack{I\in \mathcal S\\ L(I)=\mathbf L}} L(\mathcal O\cdot I).
\]

Our first main result provides an explicit description of the orbits $\mathcal O\cdot \mathbf L$. To this end, we construct an explicit map $S: \mathbb L\to \mathbb L$, the simplification map, such that, for any $\mathbf L\in \mathbb L$, we have:
\begin{itemize}
\item Either $\Ram(S(\mathbf L))<\Ram(\mathbf L)$;
\item Or $S(\mathbf L)=\mathbf L$.
\end{itemize}
We set $\mathbb L^{min}:=\left\{\mathbf L\in \mathbb L \;|\; S(\mathbf L)=\mathbf L\right\}$. Next, we define a map $\widehat{S}: \mathbb L\mapsto \mathbb L^{min}$, the full simplification map, by applying repeatedly $S$ until we reach an element of $\mathbb L^{min}$.

\begin{theorem}[Thm. \ref{thm:algorithm_gives_minimal_levels}, Prop. \ref{prop:characterization_minimal_level_data}]
The full simplification map $\widehat{S}:\mathbb L\to \mathbb L^{min}$ has the following properties:
\begin{itemize}
\item Let $\mathbf L\in \mathbb L$ be a level datum. Then $\widehat{S}(\mathbf L)$ is the unique element with minimal ramification order in the orbit $\mathcal O\cdot \mathbf L$. 
\item The image $\mathbb L^{min}$ of $\widehat{S}$ admits the following description: let $\mathbf L\in \mathbb L$, then $\mathbf L\in \mathbb L^{min}$ if and only if $\max(\mathbf L)>2$.
\end{itemize}
\end{theorem} 

However, this result does not quite imply a classification of elementary wild character varieties up to the action of basic operations, due to the following important subtlety: if $I$ is a Stokes circle of an irreducible irregular connection $(E,\nabla)$ on $\mathbb P^1$, and $A$ is a basic operation, then in general it is not true that $A\cdot I$ is a Stokes circle of $A\cdot (E,\nabla)$ (in fact this is always true if both $I$ and $A\cdot I$ are not a tame circle $\cir{0}$, but this will depend on the formal monodromy when tame circles are involved). In particular, for irregular connections on $\mathbb P^1$, the property of having a single Stokes circle is not preserved by basic operations, so that, if we start with a connection $(E,\nabla)$ with a single Stokes circle $I$ with level datum $\mathbf L$ and perform the basic operations underlying the simplification map $S$, this will in general create new extra Stokes circles along the way.  

Fortunately we can ensure that this never happens if we restrict to Stokes circles located at infinity, and basic operations which preserve this location at infinity. This still leaves us with a lot of allowed operations due to the fact that the set of Stokes circles at infinity and of slope $>1$ is globally preserved by the Fourier transform. 

We can adapt the previous constructions to this slightly modified setup. Let $\mathcal S_\infty$ denote the set of all Stokes circles at infinity. If $I\in \mathcal S_\infty$ we define
\[
\mathcal O_\infty \cdot I:= \left\{ A_m\dots A_1\cdot I \;\middle \vert\; \begin{array}{c}

 m\geq 1,  A_i \text{ basic operation, } i=1,\dots m,\\ A_k\dots A_1\cdot I\in \mathcal S_\infty \text{ for all } k=1,\dots, m
 \end{array} \right\}.
\]
For $\mathbf L\in \mathbb L$, we define
\[
\mathcal O_\infty\cdot \mathbf L:=\bigcup_{\substack{I\in \mathcal S_\infty\\ L(I)=\mathbf L}} L(\mathcal O_\infty\cdot I).
\]
Any element $\mathbf L'$ in an orbit $\mathcal O_\infty \cdot \mathbf L$ satisfies $B_{\mathbf L'}=B_{\mathbf L}$ (this is not the case for orbits $\mathcal O\cdot \mathbf L$).
In a completely similar way as before, we define a local simplification map $S_\infty:\mathbb L\to \mathbb L$, which strictly reduces the ramification order away from a subset $\mathbb L_\infty^{min}\subset \mathbb L$ of fixed elements, and a local full simplification map $\widehat{S}_\infty:\mathbb L\to \mathbb L_\infty^{min}$.

\begin{theorem}[Thm. \ref{thm:algorithm_at_infinity}] The local full simplification map $\widehat{S}_\infty:\mathbb L\to \mathbb L_\infty^{min}$ has the following properties:
\begin{itemize}
\item Let $\mathbf L\in \mathbb L$ be a level datum. Then $\widehat{S}_\infty(\mathbf L)$ is the unique element with minimal ramification order in the set $\mathcal O_\infty\cdot \mathbf L$. 
\item The image $\mathbb L^{min}_\infty$ of $\widehat{S}_\infty$ admits the following description: let $\mathbf L\in \mathbb L$, then $\mathbf L\in \mathbb L^{min}_\infty$ if and only if $\max(\mathbf L)>2$ or  $\max(\mathbf L)<1$.
\end{itemize}
\end{theorem} 

Interestingly, these results can be viewed in an elegant way, as allowing us to endow the set $\mathbb L$ with the structure of a forest, as follows: let $\Gamma^{\mathbb L}$ be the infinite graph with set of vertices $\mathbb L$, and such that $\mathbf L, \mathbf L'\in \mathbb L$, with $\mathbf L\neq \mathbf L'$, are linked by an edge in $\Gamma^{\mathbb L}$ if and only if $\mathbf L=S(\mathbf L')$, or $\mathbf L'=S(\mathbf L)$. 

Similarly, let $\Gamma^{\mathbb L}_\infty$ be the infinite graph with set of vertices $\mathbb L$, and such that $\mathbf L, \mathbf L'\in \mathbb L$, with $\mathbf L\neq \mathbf L'$, are linked by an edge in $\Gamma^{\mathbb L}_\infty$ if and only if $\mathbf L=S_\infty(\mathbf L')$, or $\mathbf L'=S_\infty(\mathbf L)$.

\begin{theorem} The graph $\Gamma^{\mathbb L}$ (respectively $\Gamma^{\mathbb L}_\infty$) has the following properties:
\begin{itemize}
\item It is acyclic, i.e. is a forest.
\item Its connected components, i.e. its trees, are the orbits $\mathcal O\cdot \mathbf L$ (respectively $\mathcal O_\infty\cdot \mathbf L$). In particular each connected component contains exactly one element of $\mathbb L^{min}$ (respectively $\mathbb L^{min}_\infty$).
\end{itemize}
\end{theorem}

Viewing each element of $\mathbb L^{min}$ (respectively $\mathbb L^{min}_\infty$) as the root of the corresponding tree, we obtain a partition of the set $\mathbb L$ of level data as an infinite union of disjoint infinite rooted trees, and the simplification map $S$ (respectively $S_\infty$) amounts to taking one step towards the root. Notice that each vertex has an infinite number of children, since there is an infinite number of ways of increasing the rank of a Stokes circle by applying a twist followed by taking the Fourier transform. A small part of the tree in $\mathbb L^{min}_\infty$ that is the connected component of the standard Lax representation $\left\{\frac{5}{2}\right\}$ of the Painlevé I moduli space is represented on Fig. \ref{fig:part_tree_Painleve_I}.\\

\begin{center}
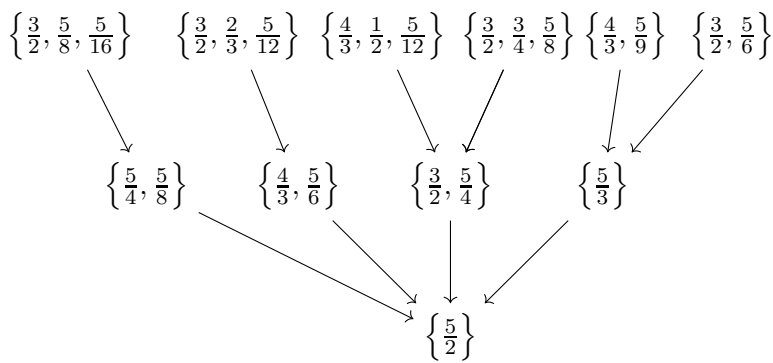
\begin{figure}[H]
\begin{tikzpicture}
\tikzstyle{vertex}=[circle,fill=black,minimum size=4pt,inner sep=0pt]
\node (A) at (0,0) {$\left\{\frac{5}{2}\right\}$};

\node (B1) at (2,2) {$\left\{\frac{5}{3}\right\}$};
\node (B2) at (0,2) {$\left\{\frac{3}{2},\frac{5}{4}\right\}$};
\node (B3) at (-2,2) {$\left\{\frac{4}{3},\frac{5}{6}\right\}$};
\node (B4) at (-4,2) {$\left\{\frac{5}{4},\frac{5}{8}\right\}$};
\node (C1) at (3.7,4) {$\left\{\frac{3}{2},\frac{5}{6}\right\}$};
\node (C2) at (2.3,4) {$\left\{\frac{4}{3},\frac{5}{9}\right\}$};
\node (C3) at (0.9,4) {$\left\{\frac{3}{2},\frac{3}{4},\frac{5}{8}\right\}$};
\node (C4) at (-0.9,4) {$\left\{\frac{4}{3},\frac{1}{2},\frac{5}{12}\right\}$};
\node (C5) at (-2.8,4) {$\left\{\frac{3}{2},\frac{2}{3},\frac{5}{12}\right\}$};
\node (C6) at (-5,4) {$\left\{\frac{3}{2},\frac{5}{8},\frac{5}{16}\right\}$};
\draw[->] (B1)--(A);
\draw[->] (B2)--(A);
\draw[->] (B3)--(A);
\draw[->] (B4)--(A);
\draw[->] (C1)--(B1);
\draw[->] (C2)--(B1);
\draw[->] (C3)--(B2);
\draw[->] (C3)--(B2);
\draw[->] (C4)--(B2);
\draw[->] (C5)--(B3);
\draw[->] (C6)--(B4);
\end{tikzpicture}
\caption{A small part of the tree corresponding to the connected component in $\Gamma_\infty^{\mathbb L}$ of the standard Lax representation $\mathbf L=\left\{\frac{5}{2}\right\}\in \mathbb L_\infty^{min}$ of the Painlevé I moduli space, which constitutes the root of the tree.  The arrows indicate the local simplification map $S_\infty$. Every vertex of the tree actually has an infinite number of children. Every vertex is expected to be an elementary Lax representation of the Painlevé I moduli space, and Thm. \ref{thm:intro_finite_number_minimal_level_data} implies that there can be no other ones (but there will also exist non-elementary Lax representations, with several Stokes circles).}
\label{fig:part_tree_Painleve_I}
\end{figure}
\end{center}

Our second main result is that there is a finite number of locally minimal level data for any possible dimension of the elementary wild character variety.

\begin{theorem}[Thm. \ref{thm:finite_number_minimal_factors}, Prop. \ref{prop:unique_non_simplifiable_small_dimension}]
\label{thm:intro_finite_number_minimal_level_data} Let $n\geq 0$ be an even integer.
\begin{itemize}
\item There is a finite number of locally minimal level data $\mathbf L\in \mathbb L^{min}_\infty$ such that $B_{\mathbf L}=n$. 
\item If $n=0, 2, 4$, there is a  unique locally minimal level datum $\mathbf L\in \mathbb L^{min}_\infty$ such that $B_{\mathbf L}=n$.
\end{itemize}
\end{theorem}

In terms of the graph $\Gamma_\infty^{\mathbb L}$, this means that the set of level data $\mathbf L$ with $B_{\mathbf L}=n$ is the union of a finite number of (infinite) trees. The list of all locally minimal level data for $B_\mathbf L$ up to 20 is given at the end of the article in Table \ref{fig:table_non-simplifiable_levels}.

This finiteness result conjecturally has consequences for the classification of elementary wild character varieties. Indeed, Möbius transformations and twists trivially induce isomorphisms of wild character varieties. For the Fourier transform, it is known that is induces a bijection between the sets of isomorphisms classes of connections on both sides \cite{malgrange1991equations, mochizuki2018stokes}, and one expects that it actually induces symplectic algebraic isomorphisms between the corresponding wild character varieties, obtained from the action of the Fourier transform of the Stokes data of the connections. However, so far explicit enough expressions for the Stokes data of the Fourier transform have only been obtained in some particular cases, including the case where the initial connection only has regular singularities \cite{balser1981reduction,malgrange1991equations,szabo2005nahm,hien2014local,dagnolo2020topological}, the case of pure gaussian type \cite{sabbah2016differential, hohl2022d_modules} or recently in joint work of the author with A. Hohl \cite{doucot2024topological} for a class of twisted wild character varieties with only one singularity at infinity, including elementary ones with a single level. In all these cases, this yields algebraic isomorphisms between the character varieties on both sides of the Fourier transform, which in all known fully explicit examples are compatible with their symplectic structures. Assuming that this holds true in the general case, our results imply:

\begin{conjecture}
Let $n$ be an even integer. There is a finite number of isomorphism classes of elementary wild character varieties of dimension $n$.
\end{conjecture}

For $n=2$, the unique locally minimal level datum $\mathbf L=\left\{\frac{5}{2}\right\}$ corresponds to the moduli space associated with the Painlevé I equation. If we adopt as in \cite{boalch2018wild} the terminology of `H3 surfaces' to refer to two-dimensional (wild) nonabelian Hodge spaces, we could say that this suggests that the Painlevé I moduli space is the unique elementary H3 surface. All its elementary Lax representations correspond to vertices of the tree with root $\mathbf L$ in the graph $\Gamma_\infty^{\mathbb L}$. Some of them are drawn on Fig. \ref{fig:part_tree_Painleve_I}.\\

Let us make some comments about related works. Note that quite similar ideas are already present in the work of Hiroe \cite{hiroe2017ramified}, although here we go further since we are also considering the non-rigid case. Notice that the situation is quite similar to what happens in the resolution of singularities of plane algebraic curves. However, while in the latter case it is always possible to obtain a resolution of singularities by application of successive blow-ups, in our context it is in general not always possible to reach a tame Stokes circle by repeated application of basic operations.

Related to our finiteness results, let us mention that recently, in a similar spirit, using the Katz-Deligne-Arinkin algorithm, Jakob counted rigid irregular connections with prescribed monodromy groups \cite{jakob2020classification,jakob2023wildly}, Sabbah obtained some finiteness results for rigid irregular connections on $\mathbb P^1$ \cite{sabbah2023remarks}, and Boalch counted untwisted fission trees \cite{boalch2024counting}.

\subsection{Contents} The article is organized as follows. In section \ref{section:reminders}, we review the necessary facts about Stokes circles, their level data, and how they transform under basic operations. In section \ref{section:simplification_general_case}, we define the simplification algorithm for a single Stokes circle in the general case and establish its properties. Finally in section \ref{section:simplification_infinity} we adapt the algorithm to the situation where Stokes circles remain at infinity, prove the finiteness of the number of locally non-simplifiable Stokes circles, and compute all locally minimal level data for small dimensions.

\subsection*{Acknowledgements} I am funded by the PNRR Grant CF 44/14.11.2022, ``Cohomological Hall Algebras of Smooth Surfaces and Applications'', led by Olivier Schiffmann. I also thank the funding of FCiências.ID, the great working environment of the Grupo de Física Matemática, where most of this work was carried out, and Philip Boalch, Andreas Hohl, Gabriele Rembado, Giordano Cotti, Giulio Ruzza, Davide Masoero, Gabriele Degano, Julian Barragán Amado, for useful discussions.

\section{Slopes and level data of Stokes circles, and basic operations}
\label{section:reminders}

In this section, we briefly recall a few facts about exponential factors/Stokes circles of irregular connections on $\Sigma=\mathbb P^1$, their slopes and level data, as how those are transformed under basic operations. We refer the reader to \cite{boalch2021topology,doucot2021diagrams, boalch2022twisted, doucot2024basic} for a more detailed treatment and proofs of the statements.

Let $a\in \Sigma$. Let $\varphi_a:\widehat \Sigma_a\to \Sigma$ be the real oriented blow-up of $\Sigma$ at $a$. The preimage $\varphi_a^{-1}(a)=:\partial_a$ is a circle parametrizing directions around $a$. Let $z_a$ be a local coordinate on $\Sigma$ around $a$ vanishing at $a$. The (local) exponential local system at $a$ is the local system of sets (i.e) the covering space $\mathcal I_a\to \partial_a$ whose sections are germs of functions on sectors which of the form
\[
q=\sum_{k=1}^s a_k z_a^{-k/r},
\]
where $r\geq 1$ is an integer, and $a_i\in \mathbb C$ for $i=1, \dots, s$ with $a_s\neq 0$. 

The smallest possible integer $r$ such that $q$ can be written in this way is the ramification order $\Ram(q)$ of $q$. Then the degree $s$ of $q$ as a polynomial in $z^{-1/r}$ is the irregularity $\Irr(q)$ of $q$. The quotient $s/r\in \mathbb Q_{\geq 0}$ is the slope $\Slope(q)$ of $q$. 

Given such a section $q$ of $\mathcal I_a$ on a sector, its connected component in $\mathcal I_a$ is a circle that we denote $\cir{q}$. The subcover $\cir{q}\to\partial_a$ has degree $r$. The set $\pi_0(\mathcal I_a)$ of \emph{Stokes circles}, or \emph{exponential factors} at $a$ is in bijection with the set of all Galois orbits of Puiseux polynomials of the form $\sum_{k=1}^s a_k z_a^{-k/r}$. We call the Stokes circle $\cir{0}$ at $a$ the \emph{tame circle} at $a$, and refer to other ones as wild Stokes circles.

The global exponential local system is the disjoint union of the local ones
\[
\mathcal I:=\sqcup _{a\in \mathbb P^1} \mathcal I_a.
\]
If $\cir{q}\in \pi_0(\mathcal I)$, we denote by $\pi(\cir{q})\in \mathbb P^1$ the point over which it lies. \\

The formal classification of meromorphic connections implies that any algebraic connection  $(E,\nabla)$ on a Zariski open susbet $\mathbb P^1\smallsetminus \{a_1,\dots, a_m\}$ of $\mathbb P^1$ defines in a canonical way a (global) \emph{irregular class} 
\[
\bm{\Theta}=n_1\cir{q_1}+\dots n_p \cir{q_p},
\]
with $\cir{q_i}\in \pi_0(\mathcal I)$, and $n_i\geq 1$,  for $i=1,\dots, p$, as well as a collection $\bm{\mathcal C}=(\mathcal C_i)_{i=1,\dots p}$  of conjugacy classes $\mathcal C_i\subset \GL_{n_i}(\mathbb C)$. We say that the Stokes circles $\cir{q_1}, \dots \cir{q_p}$ are the Stokes circles of $(E,\nabla)$, the $n_i$ are their multiplicities, and the $\mathcal C_i$ are the formal monodromies. The pair $(\bm\Theta, \bm{\mathcal C})$ constitutes the global formal data of $(E,\nabla)$. 

A connection $(E,\nabla)$ also defines \emph{modified formal data} $(\bm{\breve\Theta}, \bm{\breve{\mathcal C}})$, obtained from $(\bm\Theta, \bm{\mathcal C})$ by replacing, for every tame circle $\cir{0}_a$ at finite distance (i.e. with $a\neq\infty$) present in $\bm\Theta$ with formal monodromy $\mathcal C$, its multiplicity $n$ by $m:=\rank(A-\Id_{n})$ and replacing $\mathcal C$ by the conjugacy class $\breve{\mathcal C}\subset \GL_m(\mathbb C)$ of $A_{|{\Im(A-\Id_n)}}$, for any $A\in \GL_n(\mathbb C)$ \cite{doucot2021diagrams}.

\subsection{Level datum of a Stokes circle}

We now recall the notion of \emph{level datum} of a Stokes circle (see \cite{boalch2021topology, boalch2022twisted}).

\begin{definition}
Let $I=\cir{q}$, with $\sum_{k=1}^{s} a_i z_a^{-i/r}$ be a Stokes circle at $a$, with $r=\Ram(q)$. For $j=0, \dots, r-1$ we set $q_j:=\sum_{k=1}^s a_k e^{2k\sqrt{-1}\pi/r} z_a^{-k/r}$ the $j$-th Galois conjugate of $q$. The level datum $L(I)$ of $I$ is the set of all nonzero slopes of the differences $q_0-q_j$, for $j=0, \dots, r-1$. 
\end{definition}

In particular, a Stokes circle $I$ is untwisted if and only if its level datum $L(I)$ is empty. The notion of level datum is important because it is related to \emph{admissible deformations}: two Stokes circles are admissible deformations of each other if and only if they have the same level datum. 

There is an explicit description of the level datum in terms of the exponents appearing in the expression of $q$ as a Puiseux polynomial (this is essentially the same as the Puiseux characteristic in the theory of singularities of algebraic curves). 

\begin{proposition}[see \cite{boalch2022twisted}]
Let $I=\cir{q}$ be a Stokes circle at some $a\in \mathbb P^1$. Let $\{k_0>\dots > k_p\} \subset \mathbb Q_{>0}$ be the set of exponents of $q$, i.e. such that $q$ can be written as $q=\sum_{i=0}^p a_i z_a^{-k_i}$, with $a_i\neq 0$. Then $L(I)$ is the (possibly empty) subset $\{k_{i_0}> \dots > k_{i_l}\}$ of $\{k_0>\dots > k_p\}$ determined as follows:
\begin{itemize}
\item $i_0$ is the smallest integer $j\in \{0, \dots, p\}$ such that $k_j$ is not an integer.
\item $i_{j+1}$ is the smallest integer $j$ strictly greater than $i_j$ such that the smallest common denominator of the exponents $(k_{i_0},\dots,  k_{i_j}, k_{i_{j+1}})$ is strictly greater than the smallest common denominator of $(k_0,\dots,  k_j)$.
\end{itemize}

Conversely, for any list of positive rational numbers $\{k_0>\dots > k_p\}$ such that the the sequence of the smallest common denominators of the lists $(k_0), (k_0, k_1),\dots, (k_0,\dots, k_p)$ is strictly increasing, there exists a Stokes circle whose level datum is $\{k_0>\dots > k_p\}$.

\end{proposition}

\begin{example} Let us look at a few simple examples.
\begin{itemize}
\item If $\cir{q}$ is a Stokes circle with $r=\Ram(q)$, $s=\Irr(q)$, and $s$ and $r$ coprime, then $L(\cir{q})=\{\frac{s}{r}\}=\{\Slope(q)\}$.
\item Consider $q=z^{5/3}+z^{4/3}+z^{1/2}$. We have $\Ram(q)=6$, and the sequence of common denominators is $(3,3,6)$, so we have $L(\cir{q})=\{\frac{5}{3},\frac{1}{2}\}$.
\end{itemize}
\end{example}

\subsection{Action of basic operations}

We now discuss how basic opererations act on the ramification order, slope, and level datum of a Stokes circle. We consider the following types of basic operations, defined on the set of isomorphism classes of irreducible algebraic connections $(E,\nabla)$ on Zariski open subsets of $\mathbb  P^1$, excluding rank one connections with only  singularity of order $\leq 2$ at infinity):
\begin{itemize}
\item  The Fourier transform $F$ (see e.g. \cite{arinkin2010rigid} for its definition as an operation on irreducible connections on $\mathbb P^1)$.
\item For any unramified exponential factor $q$ at some $a\in \mathbb  P^1$, the twist $T_q$, consisting in tensoring $(E,\nabla)$ with the connection $(\mathcal O, d-dq)$ on the trivial rank one bundle, having a singularity at $a$ with $\cir{q}$ as its single Stokes circle. 
\item For every Möbius transformation $M\in \Aut(\mathbb P^1)$, we write $M\cdot (E,\nabla):=M_*(E,\nabla)$ for the operation it induces on connections. 
\end{itemize}

Every basic operation $A$ induces a well-defined bijection of the set $\pi_0(\mathcal I)$ of all possible Stokes circles, that we also denote by $A$, in a such a way that if $A$ is a twist or a Möbius transformation, if $(E,\nabla)$ has global irregular class $\bm\Theta=\sum_i n_i \;\cir{q_i}$, then the global irregular class $\bm{\Theta}'$ of $A\cdot (E,\nabla)$ satisfies:
\[
\bm{\Theta}'=\sum_i n_i\; A\cdot \cir{q_i}.
\]
On the other hand, if $A=F$ is the Fourier transform, and $(E,\nabla)$ has modified irregular class $\bm{\breve\Theta}=\sum_i n_i \; A\cdot \cir{q_i}$, then the modified irregular class $\bm{\breve\Theta}'$ of $F\cdot(E,\nabla)$ satisfies
\[
\bm{\breve\Theta}'=\sum_i n_i\; F\cdot \cir{q_i}.
\]
(this would not be true with the non-modified irregular class, and is the reason why it is convenient to introduce the modified irregular class, see \cite{doucot2021diagrams}).\\

Among the basic operations, the Fourier transform is the one with the most intricate action on the slopes and level data of Stokes circles:

\begin{lemma}[\cite{doucot2024basic}]
\label{prop:levels_Fourier_transform}
Let $I$ be a Stokes circle. We have the following: 
\begin{enumerate}
\item If $I$ i.e. of the form $\cir{\alpha z}_\infty$, with $\alpha\in\IC$, then $F\cdot I=\cir{0}_{\alpha}$.
\item If $I$ is of slope $\leq 1$ at infinity, of the form $\cir{\alpha z +q}_\infty$, with $\alpha\in\mathbb C$, and $q\neq 0$ of slope $<1$ then $F\cdot I$ is of the form $F\cdot I=\cir{\widetilde{q}}_{\alpha}$, with $\Irr(\tilde q)=s, \Ram(\tilde q)=r-s$, $L(\tilde q)=\frac{r}{r-s}L(q)$.
\item If $I$ is of slope >1 at infinity, with $\Ram I=r$, $\Irr(I)=s$, then $F\cdot I$ satisifies $\Ram(F\cdot I)=s-r$, $\Irr(F\cdot I)=s$, and $L(F\cdot I)=\frac{r}{s-r}L(I)$.
\item If $I=\cir{q}_a$ is an irregular circle at finite distance for $a\in\mathbb C=\mathbb P^1\smallsetminus\{\infty\}$, with $q\neq 0$, $\Ram(q)=r$, and $\Irr(q)=s$, then $F\cdot I$ is a circle of slope $\leq 1$, at infinity, of the form $\cir{-az+\tilde{q}}_\infty$, with $\Ram(\tilde q)=r+s, \Irr(\tilde q)=s$, and $L(\tilde{q})=\frac{r}{r+s}L(q)$.
\item If $I=\cir{0}_a$ is a tame circle at finite distance, then $F\cdot I=\cir{-az}_\infty$.
\end{enumerate}
\end{lemma}

This is obtained as a consequence of the stationary phase formula of \cite{malgrange1991equations, fang2009calculation, sabbah2008explicit}. In particular, the Fourier transform exchanges Stokes circles of slope $\leq 1$ at infinity and Stokes circles at finite distance, and preserves the subset of Stokes circles at infinity of slope $>1$.

Moreover, if $q$ is an \emph{unramified} exponential factor at $a$, applying $T_q$ has effect of shifting by $q$ the exponential  factors at $a$:

\begin{lemma}
Let $I=\cir{q'}$ be a Stokes circle. If $\pi(I)=a$ then $T_q\cdot I=\cir{q+q'}_a$, otherwise $T_q\cdot I=I$. 
\end{lemma}

In particular, this implies that we always have $L(T_q\cdot I)=L(I)$, and $\Ram(T_q\cdot I)=\Ram(I)$. However, applying $T_q$ may to $I$ may change the slope: if $\Slope(q)>\Slope(I)$, then $\Slope(T_q\cdot I)=\Slope(q)>\Slope(I)$ so the slope increases. On the other hand, if $I$ has an integer slope $s$, with leading term $a z_a^{-s}$, with $a\neq 0$, taking $q=-a z_a^{-s}$, applying $T_q$ to $I$ kills the leading term, so that $\Slope(T_q\cdot I)<\Slope(I)$. 

Finally, if $M$ is a Möbius transformation, its only effect on the irregular class $\bm{\Theta}$ is to change the position of the singularities, that is we have:

\begin{lemma}
Let $I$ be a Stokes circle and $M$ a Möbius transformation. Then $\pi(M\cdot I)=M(\pi(I))$, and $\Ram(M\cdot I)=\Ram(I)$, $\Slope(M\cdot I)=\Slope(I)$ and $L(M\cdot I)=L(I)$. 
 
 \end{lemma}

Möbius transformations will have a role to play when successively applying basic operations because they can be used to pass between singularities located at infinity and at finite distance.

\section{Simplification of Stokes circles}
\label{section:simplification_general_case}

Before discussing how to decrease the ramification order of an Stokes circle by applying basic operations, a straightforward observation is that it is always possible to increase it:

\begin{lemma}
Let $I$ be a Stokes circle. There exists a Möbius transformation $M$ such that applying $FM$ to $I$ increases its ramification order. 
\end{lemma}

\begin{proof}
One just needs to choose $M$ such that $\pi(M\cdot I)\neq \infty$. Then, if $\Slope(I)=s/r$, with $r=\Ram(I)$, one has that $\Ram(FM\cdot I)=s+r>r$.
\end{proof}

In particular, for any Stokes circle, there is no upper bound on the ramification order of the elements of its orbit $\mathcal O\cdot I$ under basic operations.

\subsection{Elementary steps and the simplification map}

Let us now discuss how to decrease the ramification order. 

\begin{definition} Let $I$ be a Stokes circle. An \emph{elementary step} on $I$ is a sequence of basic operations consisting in applying to $I$ successively:
\begin{enumerate}
\item A Möbius transformation $M$ (possibly trivial);
\item A twist $T$ at $\pi(M\cdot I)$ (possibly trivial); 
\item The Fourier transform $F$.
\end{enumerate}
\end{definition}

The level data and the slopes of $FTM\cdot I$ do not depend on the exact coefficients of the twist or the Möbius transformations. This motivates the following definition.

\begin{definition}
Let $I$ be a Stokes circle. Two elementary steps $F T M$ and $F  T' M'$ on $I$ are equivalent if they satisfy the following conditions:
\begin{itemize}
\item $\pi(M\cdot I)$ and $\pi(M'\cdot I)$ are either both at infinity, or both at finite distance, and
\item A condition of the slopes:
\begin{itemize}
\item $T M\cdot I$ and $T' M'\cdot I$ have the same slope, or
\item $T M\cdot I$ and $T' M'\cdot I$ are both at infinity, and both are of slope $\leq 1$.
\end{itemize}
\end{itemize}

\end{definition}

It follows immediately from the transformation of level data under Fourier transform that:

\begin{lemma}
If $F T M$ and $F T' M'$ are two equivalent elementary steps on $I$, then the Stokes circles $F T M\cdot I$ and $F T' M'\cdot I$ have the same level datum and the same slope.
\end{lemma}

The main observation leading to the simplication algorithm is that, up to equivalence, there is only at most one elementary step reducing the ramification order.

\begin{proposition}
\label{prop:unique_simplifying step}
Let $I$ be a twisted Stokes circle, then:

\begin{itemize}
\item There is at most one equivalence class of elementary steps strictly reducing its ramification order of $I$. 
\item There is at most one equivalence class of elementary steps preserving the ramification order of $I$. If it exists, it also preserves the level datum of $I$.
\end{itemize}
\end{proposition}

\begin{proof}
Let us write $I=\cir{q}$, $a=\pi(q)$, and ${k_0 >\dots >k_p}=L(I)$. The exponential factor $q$ may have integer exponents (strictly) higher than $k_0$ with nonzero coefficients, i.e. it is of the form

\[
q=\lambda_1 z_a^{-n_1}+\dots+ \lambda_p z_a^{-n_p} + \mu z_a^{-k_0}+ \text{ lower order terms}.
\]
with $p\geq 0$ an integer, $p=0$ corresponding to the case where there are no higher order terms, and for $p\geq 1$,  $n_1, \dots, n_p$ integers such that $n_1>\dots >n_p>k_0$, and $\lambda_1, \dots, \lambda_p\in \mathbb C^*$, $\mu\in \mathbb C^*$.

Let us analyse the possible forms of Stokes circles than can be obtained from $I$ by applying an elementary step $FT_{q'}M$.

The Stokes circles that can be obtained as $T_{q'}M\cdot I$ for some choice of $T_{q'}, M$ are of the form $\widetilde{I}=\cir{\widetilde{q}}$, with $L(\widetilde{I})=L(I)$, and

\[
\widetilde{q}=\alpha_1 z_{\widetilde{a}}^{-m_1}+\dots + \dots +\alpha_l z_{\widetilde{a}}^{-m_l} + \mu' z_{\widetilde{a}}^{-k_0}+\text{ lower order terms}
\]
where $l=0$ corresponding to the case when there are no higher order terms, and for $l\geq 1$,  $m_1 > \dots > m_l> k_0$ are integers, and $\alpha_0, \dots, \alpha_l\in\mathbb C^*$, $\mu\in \mathbb C^*$. Any set of integer exponents $m_1> \dots > m_l> k_0$ can be obtained for some appropriate choice of $T_{q'}, M$.

Now we have the following possibilities for $F\cdot \widetilde{I}$: 

\begin{itemize}
\item If $\widetilde{a}\neq \infty$, then we always have $\Ram(F \cdot \widetilde{I})>\Ram(I)$
\item Otherwise, if $\widetilde{a}=\infty$;
\begin{itemize}
\item If $1<k_0$,

\begin{itemize}
\item If $l\geq 1$ and $m_1\geq 3$ then $\Ram(F \cdot \widetilde{I})>\Ram(I)$.
\item If $l\geq 1$ and $m_1=2$ then $\Ram(F \cdot \widetilde{I})=\Ram(I)$, and $L(F \cdot \widetilde{I})=L(I)$.
\item Otherwise if $l\geq 1$ and $m_0=1$, or if $l=0$, then $\Ram(F \cdot \widetilde{I})<\Ram(I)$.
\end{itemize}

\item If $1<k_0<2$,

\begin{itemize}
\item If $l\geq 1$ and $m_1\geq 3$ then $\Ram(F \cdot \widetilde{I})>\Ram(I)$.
\item If $l\geq 1$ and $m_1=2$ then $\Ram(F \cdot \widetilde{I})=\Ram(I)$, and $L(F \cdot \widetilde{I})=L(I)$.
\item Otherwise, if $l=0$, then $\Ram(F \cdot \widetilde{I})<\Ram(I)$.
\end{itemize}

\item If $k_0>2$, 

\begin{itemize}
\item If $l\geq 1$, then $m_1\geq 3$, and $\Ram(F \cdot \widetilde{I})>\Ram(I)$.
\item Otherwise, if $l=0$, then $\Ram(F \cdot \widetilde{I})>\Ram(I)$.
\end{itemize}

\end{itemize}

\end{itemize}

We see that for any value of the greatest level $k_0$, there is at most one of the above subcases where the ramification order is strictly reduced. It follows from the definition of equivalence of elementary steps that each case corresponds to exactly one equivalence class of elementary steps, which concludes the proof. 
\end{proof}

We say that an elementary step strictly reducing the ramification order of $I$ is a \emph{simplifying step} for $I$. The proposition shows that if such a step exists, there is a unique equivalence class of simplifying steps, which we we call \emph{the} simplifying step of $I$. 

Let $\mathbb L$ denote the set of all possible level data. If $I$ is an Stokes circle, and $FTM$ is a simplifying step for $I$, the level datum of $FTM\cdot I$ only depends on $L(I)$, so we have a well-defined simplification map
\[ S: \mathbb L\mapsto \mathbb L,
\]
where we set $S(\mathbf L)=\mathbf L$ when $\mathbf L$ is empty, or doesn't admit a simplifying step.

Let $\mathbb L^{min}\subset \mathbb L$ be the set of level data that don't admit a simplifying step, i.e. that are fixed points of $S$. From the proof of the proposition, we immediately have the following characterization:

\begin{proposition}
\label{prop:characterization_minimal_level_data}
Let $\mathbf L$ be a level datum. Then $\mathbf L\in \mathbb L^{\min}$ if and only if $\mathbf L$ is empty, or its highest level $k$ satisfies $k>2$. 
\end{proposition}

\subsection{Full simplification}

Now, we obtain a full simplification algorithm simply by repeatedly appying $S$.  

\begin{definition} Let $\mathbf L\in \mathbb L$ be a level datum. The full simplification algorithm consists in, starting from $\mathbf L$, applying iteratively $S$ as long as a simplifying step exists (i.e.  $S$ reduces the ramification order), and stop as soon as there is no simplifying step.
\end{definition}

The algorithm always terminates since a simplifying step strictly reduces the rank. If $\mathbf L$ is a level datum, let $\widehat{S}(\mathbf L)\in \mathbb L^{min}$ denote the level datum obtained when the algorithm terminates. This defines a full simplification map
\[
\widehat{S}:\mathbb L\to \mathbb L^{min}.
\]

We now show that the simplification algorithm always gives us a unique level datum with minimal ramification order in the orbit $\mathcal O\cdot \mathbf L$. 

\begin{theorem}
\label{thm:algorithm_gives_minimal_levels}
Let $\mathbf L\in \mathbb L$ be a level datum. Then its simplification $\widehat{S}(\mathbf L)$ is an element with minimal ramification order in the orbit $\mathcal O\cdot \mathbf L$. Furthermore, this minimal element is unique.
\end{theorem}

In particular, the theorem implies that Stokes circles which are non-simplifiable are actually of minimal ramification order in their orbit, which justifies the notation $\mathbb L^{min}$.

\begin{proof}
Let us consider a Stokes circle $I$ with level datum $\widehat{S}(\mathbf L)$, and a composition $B=B_n\dots B_1$ of basic operations such that $\Ram(B\cdot I)\leq \Ram(\widehat{S}(\mathbf L))$. We want to show that $L(B\cdot I)=L(I)=\widehat{S}(\mathbf L)$. Since Möbius transformations and twists commute, the composition two Möbius transformations is a Möbius transformation and the composition of two twists is a twist, up to reordering and regrouping the $B_i$, we can always write $B$ as a composition of elementary steps $E_1,\dots, E_m$. 

Let us write $I_0=I$, and for $i=1,\dots, m$, $I_i:=E_i\dots E_1\cdot I$. Also, we set $r_i:=\Ram(I_i)$ for $i=0, \dots m$.

Using prop. \ref{prop:unique_simplifying step}, we show by induction on the length $m$ that $L(I_m)=L(I)$. 

First, if $m=1$, since $\widehat{S}(\mathbf L)$ is the stopping point of the algorithm, there is no simplifying step, so $r_1=r_0$, the step $E_1$ preserves the ramification order, and in turn $L(I_1)=L(I_0)=L(I)$.

Now assume $m\geq 2$. If the sequence $r_0, \dots, r_m$ is constant, then we have $L(I_0)=\dots=L(I_m)$ and we are done. Otherwise, since $\widehat{S}(\mathbf L)$ is the stopping point of the algorithm, if $i$ is the smallest integer $j\geq 1$ such that $r_j\neq r_0$, we have $r_i>r_0$, so in particular the maximum $r_{max}$ of $r_0, \dots, r_m$ satisfies $r_{max}>r_0$, and also $r_{max}>r_m$.

In turn there exist indices $i,j\in \{0, m\}$ with $j\geq i+2$ such that we have
\[
r_i < r_{max}=r_{i+1}=\dots=r_{j-1}> r_j.
\]
But since by Prop. \ref{prop:unique_simplifying step} we have $L(I_{i+1})=\dots=L(I_{j-1})$, and then by uniqueness of the simplifying step we have $L(I_i)=L(I_j)$. 

This means that $L(I_m)$ can also be reached from $L(I_0)$ by a sequence $E_1, \dots E_i, \widetilde{E}_{i+1}, E_{j+1}, \dots E_m$ of elementary steps of length $<m$, so by the induction hypothesis we have $L(I_m)=L(I_0)$.   
\end{proof}

As an application, we can characterize when it is possible to fully simplify a Stokes circle to obtain a tame circle.

\begin{definition}
A Stokes circle $I$ is \emph{tamable} if it can be brought to a tame circle $\cir{0}$ at some $a\in\mathbb P^1$ by successive application of basic operations, i.e. if $\emptyset \in \mathcal O\cdot I$.
\end{definition}

\begin{corollary}
Let $I$ be a Stokes circle with level datum $\mathbf L$. Then $I$ is tamable if and only if $\widehat{S}(\mathbf L)=\emptyset$.
\end{corollary}

\begin{example}
Consider the level datum $\mathbf L=\{\frac{3}{5},\frac{2}{15},\frac{2}{105}\}=\{\frac{63}{105},\frac{14}{105},\frac{2}{105}\}$. Running the algorithm gives 
\begin{align*}
\left\{\frac{3}{5},\frac{2}{15},\frac{2}{105}\right\}=\left\{\frac{63}{105},\frac{14}{105},\frac{2}{105}\right\} &\overset{S}{\longmapsto} \left\{\frac{63}{42}, \frac{14}{42}, \frac{2}{42}\right\}=\left\{\frac{3}{2}, \frac{1}{3}, \frac{1}{21}\right\}\\
&\overset{S}{\longmapsto} \left\{\frac{14}{21}, \frac{2}{21}\right\}=\left\{\frac{2}{3}, \frac{2}{21}\right\}\\
  &\overset{S}{\longmapsto} \left\{\frac{2}{7}\right\}\\
  &\overset{S}{\longmapsto} \left\{\frac{2}{5}\right\}\\
  &\overset{S}{\longmapsto} \left\{\frac{2}{3}\right\}\\
  &\overset{S}{\longmapsto} \emptyset, 
\end{align*}
which shows that $\mathbf L$ is tamable.
\end{example}

The Katz-Deligne-Arinkin algorithm implies that if $(E,\nabla)$ is an irreducible rigid irregular connection on $\mathbb P^1$, then any of its Stokes circles is tamable. In turn we have:

\begin{corollary}
Let $(E,\nabla)$ be an irreducible rigid irregular connection on $\mathbb P^1$, $I$ a Stokes circle of $(E,\nabla)$, and $\mathbf L=L(I)$. Then $\widehat{S}(\mathbf L)=\emptyset$.
\end{corollary}

Equivalently, if $\widehat{S}(\mathbf L)\neq\emptyset$, then no irreducible rigid irregular connection on $\mathbb P^1$ admits a Stokes circle with level datum $\mathbf L$.

As a further example, let us discuss the particular case of Stokes circles with just one level. 

\begin{lemma}
Let $\mathbf L=\{\frac{s}{r}\}$, with $r,s$ coprime, be a level datum consisting of a single level. We have:
\begin{itemize}

\item If $\frac{s}{r}<1$, then
\[
S(\mathbf L)=\left\lbrace
\begin{array}{cc}
\emptyset  &\text{ if } r=s+1, \\
\left\{\frac{s}{r-s}\right\}  &\text{ otherwise. } \\
\end{array}
\right.
\]

\item If $1<\frac{s}{r}<2$, then
\[
S(\mathbf L)=\left\lbrace
\begin{array}{cc}
\emptyset  &\text{ if } s=r+1, \\
\left\{\frac{s}{r}\right\}  &\text{ otherwise. } \\
\end{array}
\right.
\]

\item If $\frac{s}{r}>2$, then
\[
S(\mathbf L)=
\left\{\frac{s}{s-r}\right\}. 
\]
\end{itemize}
\end{lemma}

\begin{proof}
This follows in a straightforward way from the proof of prop. \ref{prop:unique_simplifying step}, and the formulas of prop. \ref{prop:levels_Fourier_transform} for the level datum of the Fourier transform. More precisely, if $\frac{s}{r}<1$ and $I$ is a Stokes circle with slope $\frac{s}{r}$, and $E$ is a simplifying step for $I$, then $E\cdot I$ is a Stokes circle factor of slope $\frac{s}{r-s}$. There are then two possibilities: either this slope is an integer, which happens when $r=s+1$, so $E\cdot I$ is unramified, i.e. $L(E\cdot I)=\emptyset$; or otherwise $\frac{s}{r-s}$ is not an integer, and is the unique level of $E\cdot I$. The reasoning is similar for the cases $1<\frac{s}{r}<2$ and $\frac{s}{r}>2$.
\end{proof}

\begin{corollary}
Let $\mathbf L=\{\frac{s}{r}\}$, with $r,s$ coprime, be a level datum consisting of a single level. We have: 
\begin{itemize}
\item If $r\equiv \pm 1 \mod s$, then $\widehat{S}(\mathbf L)=\emptyset$.
\item Otherwise, $\widehat{S}(\mathbf L)=\{\frac{s}{k}\}$, where $k$ is the unique integer such that $1<k<\frac{s}{2}$, and $r\equiv \pm k \mod s$.
\end{itemize}
\end{corollary}

\begin{proof}
This follows immediately from the previous lemma, by applying $S$ repeatedly.
\end{proof}

In particular, $\mathbf L$ is tamable if and only if  $r\equiv \pm 1 \mod s$, which agrees with \cite[Thm 0.3]{hiroe2017ramified}.

\subsection{Local simplification with the singularity fixed at infinity}

If $(E,\nabla)$ is an irreducible connection on $\mathbb P^1$, neither the number of Stokes circles of its global irregular class $\bm{\Theta}$, nor the number of Stokes circles of its modified global irregular class $\bm{\breve\Theta}$ (which is also the number of nodes of the core diagram $\Gamma_c(E,\nabla)$ associated to $(E,\nabla)$ by the construction of \cite{doucot2021diagrams}), are preserved by basic operations: the first number is preserved by Möbius transformations and twists at one of the singularities, but not the Fourier transform, and the converse is true for the second number. As a consequence, for irreducible connections on $\mathbb P^1$, the property of having a single Stokes circle, or of having a core diagram consisting of a single vertex, are not preserved by basic operations. 

This is also reflected by the fact that if $\mathbf L$ is a level datum, $I$ is a Stokes circle with level datum $\mathbf L$, and $\mathbf L'=L(F\cdot I)$, in general we don't have $B_{\mathbf L}=B_{\mathbf L'}$. This is true however when both $I$ and $F\cdot I$ are at infinity. 

For these reasons, in the rest of the article, we will consider connections with a single Stokes circle on $\mathbb P^1$ located at infinity, and only basic operations which preserve this property. Given a Stokes circle $I$ at infinity, we allow for the following \emph{basic operations at infinity}:
\begin{itemize}
\item If $\Slope(I)>1$, the Fourier transform $F$. In that case we also have $\Slope(F\cdot I)>1$.
\item Twists at infinity.
\end{itemize}

We can adapt the constructions of elementary steps and the simplification algorithm of the previous section. 

\begin{definition} Let $I$ be a twisted Stokes circle at infinity. An \emph{local elementary step} on $I$ is a sequence of  basic operations at infinity consisting in applying to $I$ successively:
\begin{enumerate}
\item A twist $T$ at infinity (possibly trivial), such that $\Slope(T\cdot I)>1$.
\item The Fourier transform $F$.
\end{enumerate}
Two local elementary steps $F T$ and $F  T'$ on $I$ are \emph{equivalent} if $T\cdot I$ and $T'\cdot I$ have the same slope.
\end{definition}

In a similar way as before, we have:

\begin{lemma}
If $F T$ and $F T'$ are two local equivalent elementary steps on $I$, then the exponential factors $F T\cdot I$ and $F T'\cdot I$ have the same level datum and the same slope.
\end{lemma}

We also have an analogue of prop. \ref{prop:unique_simplifying step}:

\begin{proposition}
\label{prop:unique_simplifying step_infinity}
Let $I$ be a twisted Stokes circle, such that its largest level is $>1$, then:

\begin{itemize}
\item There is at most one equivalence class of local elementary steps strictly reducing the ramification order of $I$. 
\item There is at most one equivalence class of local elementary steps preserving the ramification order of $I$. If it exists, it also preserves the level datum of $I$.
\end{itemize}
\end{proposition}

\begin{proof}
The proof is completely similar to the one of \ref{prop:unique_simplifying step}.
\end{proof}

We say that a local elementary step strictly reducing the ramification order of $I$ is a local simplifying step for $I$. If such a step exists, there is a unique equivalence class of local simplifying steps, which we we call \emph{the} local simplifying step of $I$. 

If $I$ is a Stokes circle at infinity, and $FT$ is a local simplifying step for $I$, the level datum of $FT\cdot I$ only depends on the level datum $L(I)$, so we have a well-defined local simplification map
\[ S_{\infty}: \mathbb L\mapsto \mathbb L,
\]
where we set $S_\infty(\mathbf L)=\mathbf L$ when $\mathbf L$ is empty, or doesn't admit a local simplifying step.

Let $\mathbb L^{min}_{\infty}\subset \mathbb L$ be the set of level data that don't admit a local simplifying step, i.e. that are fixed points of $S_\infty$. We have the following characterization:

\begin{proposition}
\label{prop:characterization_non_simplifiable_factor_infinity}
Let $\mathbf L$ be a level datum. Then $\mathbf L\in \mathbb L^{\min}_{\infty}$ if and only if $\mathbf L$ is empty, or its highest level $k$ satisfies $k<1$ or $k>2$.
\end{proposition}

Again, we obtain a local simplification algorithm for Stokes circles at infinity simply by repeatedly applying $S_\infty$.  

\begin{definition} Let $\mathbf L\in \mathbb L$ be a level datum. The local simplification algorithm at infinity consists in, starting from $\mathbf L$, applying iteratively $S_\infty$ as long as a simplifying step exists (i.e.  $S_\infty$ reduces the ramification order), and stop as soon as there is no local simplifying step.
\end{definition}

If $\mathbf L$ is a level datum, let $\widehat{S}_\infty(\mathbf L)\in \mathbb L^{min}_\infty$ denote the level datum obtained when the algorithm terminates. This defines a local simplification map
\[
\widehat{S}_{\infty}:\mathbb L\to \mathbb L^{min}_\infty.
\]

Finally, as before, the local full simplification algorithm always gives us a unique level datum with minimal ramification order reachable by basic operations preserving at each step the location at infinity:

\begin{theorem}
\label{thm:algorithm_at_infinity}
Let $\mathbf L$ be a level datum. Then $\widehat{S}_\infty(\mathbf L)$ is an element with minimal ramification order in the orbit $\mathcal O_\infty\cdot \mathbf L$. Furthermore, this minimal element is unique.
\end{theorem}

\begin{proof}
The proof is completely analogous to the one of theorem \ref{thm:algorithm_gives_minimal_levels}.
\end{proof}

\begin{example}
Consider $\mathbf L=\left\{\frac{4}{3},\frac{4}{9},\frac{1}{8}\right\}$. Its number of loops is $k=0$. The algorithm gives
\begin{align*}
\mathbf L=\left\{\frac{4}{3},\frac{4}{9},\frac{1}{8}\right\} = \left\{\frac{96}{72},\frac{32}{72},\frac{9}{72}\right\}
                & \overset{S_\infty}{\longmapsto} \left\{\frac{32}{24}, \frac{9}{24}\right\}=\left\{\frac{4}{3},\frac{3}{8}\right\}\\
                &\overset{S_\infty}{\longmapsto} \left\{\frac{9}{8}\right\}\\
                &\overset{S_\infty}{\longmapsto} \emptyset.        
\end{align*}

Consider $\mathbf L'=\left\{\frac{3}{2},\frac{4}{7},\frac{1}{4}\right\}$, with $k=2$ loops. The algorithm gives
\begin{align*}
\mathbf L'=\left\{\frac{3}{2},\frac{4}{7},\frac{1}{4}\right\}= \left\{\frac{42}{28},\frac{16}{28},\frac{7}{28}\right\}
                & \overset{S_\infty}{\longmapsto} \left\{\frac{16}{14}, \frac{7}{14}\right\}=\left\{\frac{8}{7},\frac{1}{2}\right\}\\
                &\overset{S_\infty}{\longmapsto} \left\{\frac{7}{2}\right\}.       
\end{align*}
\end{example}

\subsection{The forest of level data}

We can interpret the previous results in a nice way if we define a graph structure on the set $\mathbb L$ of level data, by linking by an edges level data related by a (local) simplifying step.

More precisely, let us to endow the set $\mathbb L$ with the structure of a graph, as follows: let $\Gamma^{\mathbb L}$ be the infinite graph with set of vertices $\mathbb L$, and such that $\mathbf L, \mathbf L'\in \mathbb L$, with $\mathbf L\neq \mathbf L'$, are linked by an edge in $\Gamma^{\mathbb L}$ if and only if $\mathbf L=S(\mathbf L')$, or $\mathbf L'=S(\mathbf L)$. 

Similarly, $\Gamma^{\mathbb L}_\infty$ be the infinite graph with set of vertices $\mathbb L$, and such that $\mathbf L, \mathbf L'\in \mathbb L$, with $\mathbf L\neq \mathbf L'$, are linked by an edge in $\Gamma^{\mathbb L}_\infty$ if and only if $\mathbf L=S_\infty(\mathbf L')$, or $\mathbf L'=S_\infty(\mathbf L)$.

\begin{theorem} The graph $\Gamma^{\mathbb L}$ (respectively $\Gamma^{\mathbb L}_\infty$) has the following properties:
\begin{itemize}
\item It is acyclic, i.e. is a forest.
\item Its connected components, i.e. its trees, are the orbits $\mathcal O\cdot \mathbf L$ (respectively $\mathcal O_\infty\cdot \mathbf L$). In particular each connected component contains exactly one element of $\mathbb L^{min}$ (respectively $\mathbb L^{min}_\infty$).
\end{itemize}
\end{theorem}

\begin{proof}
For the first assertion, by construction of $\Gamma^{\mathbb L}$ (resp. $\Gamma^{\mathbb L}_\infty$) there are no loops, nor cycles of length 2 (double edges). If $c=(\mathbf L_1, \dots, \mathbf L_m)$ is a cycle in $\Gamma^{\mathbb L}$ (resp. $\Gamma^{\mathbb L}_\infty$) of length $m\geq 3$, let $\mathbf L_k$ be an element with maximal ramification order in $c$. Then either we have $\Ram(\mathbf L_k)>\Ram(\mathbf L_{k-1})$ and $\Ram(\mathbf L_k)>\Ram(\mathbf L_{k+1})$ (we consider indices modulo $m$), or we have $\Ram(\mathbf L_k)=\Ram(\mathbf L_{k+1})$ or $\Ram(\mathbf L_k)=\Ram(\mathbf L_{k-1})$. All cases are forbidden by Prop. \label{prop:unique_simplifying step} (resp. Prop. \ref{prop:unique_simplifying step_infinity}), so we have a contradiction. The second assertion follows then immediately from Thm. \ref{thm:algorithm_gives_minimal_levels} (resp. Thm \ref{thm:algorithm_at_infinity}). 
\end{proof}

If we view each element of $\mathbb L^{min}$ (respectively $\mathbb L^{min}_\infty$) as the root of the tree constituted by its connected component in the graph, we obtain a partition of the set $\mathbb L$ of level data as an infinite union of disjoint rooted infinite trees, and simplification map $S$ (respectively $S^\infty$) amounts to taking one step towards the root, and the full simplification map $\widehat{S}$ (resp. $\widehat{S}_\infty$) sends any vertex to the root of the tree it belongs to. 

Furthermore, any vertex $\mathbf L$ has an infinite number of children in both graphs, since (even without using Möbius transformations) there is an infinite number of ways of increasing the rank of a Stokes circle by applying a twist followed by taking the Fourier transform.\\

\section{Finiteness of minimal level data at infinity}
\label{section:simplification_infinity}

Let $I$ be a twisted Stokes circle at infinity, with level datum $\mathbf L$. Let $r$ be the ramification order of $I$, and let us write  $\mathbf L=\left\{\frac{s_0}{r},\dots,\frac{s_p}{r}\right\}$, with $s=s_0>\dots>s_p$. Then the diagram associated by \cite{doucot2021diagrams} to any connection with irregular class $I$ consists of one vertex, with a number of loops equal to $\frac{B_\mathbf{I}}{2}$, where $B_I$ is given by the formula

\begin{equation}
B_I=B_\mathbf{L}=(r-(s_0,r))s_0+\dots+((s_0,\dots,s_{p-1},r)-(s_0,\dots,s_{p},r))s_p-r^2+1,
\end{equation}
where $(\cdot, \dots, \cdot)$ denotes the greatest common divisor. One can indeed show that $B_\mathbf{L}$ is always an even integer. Furthermore, the expected dimension of the elementary wild character variety $\mathcal M_B(\mathbf L)$ is given by 
\begin{equation}
\dim \mathcal M_B(\mathbf L)=B_\mathbf{L}.
\end{equation}

\begin{lemma}
Let $\mathbf L$ be a nonempty level datum, and $k$ its largest level. If $k<1$, then $B_{\mathbf L}<0$.
\end{lemma}

\begin{proof}
We have 
\begin{align*}
B_{\mathbf L}&=(r-(s_0,r))s_0+\dots+((s_0,\dots,s_{p-1},r)-(s_0,\dots,s_{p},r))s_p-r^2+1\\
&\leq  (r-(s_0,r))s_0+\dots+((s_0,\dots,s_{p-1},r)-(s_0,\dots,s_{k},r))s_0-r^2+1\\
&\leq sr-r^2+1\\
&\leq -r +1\\
&\leq -1,
\end{align*}
where to pass to the second line we use that $s_i<s_0$ for $i=1,\dots, p$, and to pass from the third to the fourth line we use that $s<r$ since $k=\frac{s}{r}<1$.
\end{proof}

\begin{theorem}
\label{thm:finite_number_minimal_factors}
For any integer $n\in \mathbb Z_{\geq 0}$, the number of minimal level data $\mathbf L\in \mathbb L^{min}_\infty$ such that $B_{\mathbf L}=2n$ is finite.
\end{theorem}

\begin{proof}
Let $n\geq 0$ and integer, and $\mathbf L\in \mathbb L^{min}_\infty$ such that $B_{\mathbf L}=2n$. By the previous lemma, the largest level $k$ of $\mathbf L$ satisfies $k>1$. Since $\mathbf L$ is not simplifiable, by Prop. \ref{prop:characterization_non_simplifiable_factor_infinity} this implies $k>2$. 
Instead of reasoning with $\mathbf L$, it will be slightly more convenient to work with its Fourier transform: if $I$ is a Stokes circle with level datum $\mathbf L$ and slope $k$, then $L':=L(F\cdot I)$ only depends on $L$. By invariance of the diagram under Fourier transform \cite{doucot2021diagrams}, it satisfies $B_{\mathbf L'}=B_{\mathbf L}=2n$, and its largest level $k'$ satisfies $1<k'<2$. 

Let $r$ be the ramification order of $\mathbf L'$, and let us write $k'=\frac{s}{r}$, and $\mathbf L'=\left\{\frac{s_0}{r},\dots,\frac{s_p}{r}\right\}$, with $s=s_0>\dots>s_p$.

We have 
\begin{align*}
2n=B_{\mathbf L}&=(r-(s,r))s+\dots+((s_0,\dots,s_{p-1},r)-(s_0,\dots,s_{p},r))s_p-r^2+1\\
&\geq (r-(s,r))s-r^2+1=:m(s,r).
\end{align*}
It will be sufficient to prove that the condition $m(s,r)\leq 2n$ leaves a finite number of possibilities for the pair $(s,r)$, since for each such possibility there is a finite number of possible sequences $s_0>\dots >s_p$. 

Let us set $d:=(s,r)\geq 1$, and write $s=\sigma d$, $r=\rho d$, with $\sigma, \rho$ coprime. We notice that $s,r,d$ satisfy some inequalities. First, since $\frac{s}{r}<2$, we have $s<2r$, and in turn (since both $s$ and $r$ are multiples of $d$) we have
\[ s\leq 2r-d.
\]
Second, since $\frac{s}{r}<2$, we have $s>r$ so $s\geq r+d$. But we have better: $s=r+d$ is excluded, indeed in that case we would have $r=\rho d$, $s=(\rho+1)d$, so $k'=\frac{s}{r}=\frac{\rho+1}{\rho}$, which by the formula for the slopes of the Fourier transform implies $k=\rho$, and we have a contradiction since $k$ is a level of $\mathbf L$ hence cannot be an integer. So we have $s>r+d$, hence
\[
s\geq r+2d.
\]
This implies that $s$ can be written $s=r+ld$, with $2\leq l\leq \rho-1$.
Let us use this to derive from the inequality $m(s,r)\leq 2n$ a simpler inequality satisfied by $s,r,d$. 
\begin{align*}
m(s,r)&=(r-d)s-r^2+1\\
&= (r-d)(r+ld)-r^2+1\\
&=(l-1)rd -ld^2+1\\
&=\left((l-1)\rho -l\right) d^2+1
\end{align*}
We have $0<l<\rho$ so $(l-1)\rho-l\geq (l-2)\rho$. Hence we obtain the inequality
\[
2n \geq m(s,r)\geq (l-2)\rho d^2+1> (l-2)rd.
\]
If $l-2>0$, this implies $rd<2n$, hence $r<2n$, which leaves a finite number of possibilities for $r$, and in turn for $s$. 

Therefore (since $l\geq 2$) it only remains to consider the case  $l=2$, that is $s=r+2d$. In that case we have
\[
2n\geq m(s,r)=(\rho-2)d^2+1.
\]
If $\rho>2$, this leaves a finite number of possibilities for the pair $(\rho,d)$, and we are done.

But clearly one cannot have $\rho=1$, since this would mean $r=d$, and $r$ divides $s$, which is not possible since the level $k'=\frac{s}{r}$ cannot be an integer. The case $\rho=2$ is also impossible. Indeed, this means $r=2d$, and $s=r+2d=4d$, but then $s=2r$ and we have a contradiction. This concludes the proof.
\end{proof}

To concretely implement a computer search of minimal level data, it is convenient to notice that we have a simple bound satisfied by $r$ in all cases:

\begin{lemma}
\label{lemma:bound_ramification}
Let $\mathbf L\in \mathbb L^{min}_\infty$ be a nonempty minimal level datum at infinity such that $B_{\mathbf L}=2n \geq 0$. Let $k>2$ be its largest level, and let us write $k=\frac{s}{s-r}$ with $r-s$ the ramification order of $\mathbf L$, and $r<s<2r$. Then the integer $r$ satisfies the inequality
\[ r< 6n.
\]
\end{lemma}

\begin{proof}
Keeping the notations of the previous proof, in the case $l>2$, we have $r<2n$, so the inequality $r<6n$ is also satisfied. In the case $l=2$, we showed that $\rho\geq 3$, i.e. $r\geq 3d$,or $d\leq \frac{r}{3}$. In turn we have 
\[
2n\geq m(s,r)=(r-2d)d+1\geq \frac{r}{3} d+1> \frac{r}{3},
\]
so $r<6n$.
\end{proof}

For small dimensions $2n=0,2,4$, it turns out that there is actually a unique minimal level datum:

\begin{proposition}
\label{prop:unique_non_simplifiable_small_dimension}
 Let $\mathbf L\in \mathbb L^{min}_\infty$. Then
\begin{itemize}
\item $B_{\mathbf L}=0$ if and only if $\mathbf L=\emptyset$.
\item $B_{\mathbf L}=2$ if and only if $\mathbf L=\left\{\frac{5}{2}\right\}$.
\item $B_{\mathbf L}=4$ if and only if $\mathbf L=\left\{\frac{7}{2}\right\}$.
\end{itemize}
\end{proposition}

\begin{proof}

We keep all the notations from the proof of \ref{thm:finite_number_minimal_factors}. 
Let us assume that $B_\mathbf L\in \{0, 2, 4\}$. We first show that $\mathbf L$ (or equivalently $\mathbf L'$) only has one level. Indeed, if the number of levels of $\mathbf L'$ is $p\geq 2$, then we have $(s,r)=d\geq 2$, and from the proof of \ref{thm:finite_number_minimal_factors}, in all cases we have the inequality
\[
B_{\mathbf L'}\geq d^2+1\geq 5,
\]
so we have a contradiction. 
So $\mathbf L'$ is either empty, or only has one level, i.e. $\mathbf L'=\left\{\frac{s}{r}\right\}$, with $s, r$ coprime, $2\leq r<s<2r$, and we have:

\[
B_{\mathbf L'}=(r-1)(s-r-1)\geq r-1,
\]
and in particular $r\leq B_{\mathbf L'}+1$. In turn $\mathbf B_{\mathbf L'}=0$ is impossible if $\mathbf L'\neq \emptyset$. If $B_{\mathbf L'}=2$, then $r\leq 3$, and it is easily seen that the only possibility for $(r,s)$ is $\mathbf L'=\left\{\frac{5}{3}\right\}$, which corresponds to $\mathbf L=\left\{\frac{5}{2}\right\}$. Finally, of $B_{\mathbf L'}=4$, then $r\leq 5$, and the only possibility for $(r,s)$ is $\mathbf L'=\left\{\frac{7}{5}\right\}$, which corresponds to $\mathbf L=\left\{\frac{7}{2}\right\}$.
\end{proof}

For higher dimensions, there are several minimal level data, which we conjecture correspond to non-isomorphic wild character varieties. Using our bound from lemma \ref{lemma:bound_ramification}, we can do a computer search to find all minimal level data for any number of loops. The minimal level data for up to ten loops are written down in table \ref{fig:table_non-simplifiable_levels}. We notice that the sequence of the first numbers of minimal level data at infinity doesn't seem to match with any known integer sequence listed on \href{https://oeis.org/}{oeis.org}.

\newpage

\begin{center}
\begin{figure}[h]
\begin{tabular}{|c|c|c|}
  \hline
  Number of loops &  Minimal level data & Number of minimal level data \\
  \hline
  0 & $\emptyset$ & 1 \\
  \hline
  1 & $\left\{\frac{5}{2}\right\}$ & 1\\
  \hline
  2 & $\left\{\frac{7}{2}\right\}$ & 1 \\
  \hline
  3 & $\begin{array}{c}
       \left\{\frac{9}{2}\right\}\\
       \left\{\frac{7}{3}\right\}\\
       \left\{\frac{5}{2},\frac{1}{4}\right\}
      \end{array}$
      & 3 \\
      \hline
  4 & $\begin{array}{c}
       \left\{\frac{11}{2}\right\}\\
       \left\{\frac{8}{3}\right\}\\
       \left\{\frac{5}{2},\frac{3}{4}\right\}
      \end{array}$
      & 3 \\
      \hline
   5 & $\begin{array}{c}
       \left\{\frac{13}{2}\right\}\\
       \left\{\frac{5}{2},\frac{5}{4}\right\}
      \end{array}$
      & 2 \\
      \hline
 6 & $\begin{array}{c}
       \left\{\frac{15}{2}\right\}\\
       \left\{\frac{10}{3}\right\}\\
       \left\{\frac{9}{4}\right\}\\
       \left\{\frac{5}{2},\frac{1}{6}\right\}\\
       \left\{\frac{5}{2},\frac{7}{4}\right\}
      \end{array}$
      & 5 \\
      \hline   
 7 & $\begin{array}{c}
       \left\{\frac{17}{2}\right\}\\
       \left\{\frac{11}{3}\right\}\\
       \left\{\frac{7}{2}, \frac{1}{4}\right\}\\
       \left\{\frac{5}{2},\frac{1}{3}\right\}\\
       \left\{\frac{5}{2},\frac{9}{4}\right\}
      \end{array}$
      & 5 \\
      \hline
 8 & $\begin{array}{c}
       \left\{\frac{19}{2}\right\}\\
       \left\{\frac{7}{2}, \frac{3}{4}\right\}\\
       \end{array}$
      & 2 \\
      \hline  
 9 & $\begin{array}{c}
       \left\{\frac{21}{2}\right\}\\
       \left\{\frac{13}{3}\right\}\\
       \left\{\frac{11}{4}\right\}\\
       \left\{\frac{7}{2},\frac{5}{4}\right\}\\
       \left\{\frac{5}{2},\frac{2}{3}\right\}
      \end{array}$
      & 5 \\
      \hline 
 10 & $\begin{array}{c}
       \left\{\frac{23}{2}\right\}\\
       \left\{\frac{14}{3}\right\}\\
       \left\{\frac{11}{5}\right\}\\
       \left\{\frac{7}{2},\frac{7}{4}\right\}\\
       \left\{\frac{5}{2},\frac{1}{8}\right\}\\
       \left\{\frac{5}{2},\frac{5}{6}\right\}
      \end{array}$
      & 6 \\
      \hline

  \hline
\end{tabular}
\caption{Table of minimal level data at infinity, corresponding to one-vertex diagrams with up to 10 loops.}
\label{fig:table_non-simplifiable_levels}
\end{figure}
\end{center}

\end{document}